\documentclass[a4paper,12pt,leqno]{amsart}
\usepackage{graphicx}
\usepackage[applemac]{inputenc}
\usepackage{selinput}
\usepackage{color}
\SelectInputMappings{
adieresis={ä},
ntilde={ñ}
}
\usepackage[spanish, english]{babel}
\usepackage{amssymb}
\usepackage{mathabx}
\usepackage{mathrsfs}
\usepackage{amsfonts,amscd,amsmath}
\usepackage{amssymb,latexsym}
\usepackage{enumitem} \numberwithin{equation}{section}

\usepackage{hyperref}

\def\irr#1{{\Irr}(#1)}

\def\zent#1{{\bf Z}(#1)}
\def\syl#1#2{{\rm Syl}_{#1}(#2)}

\def\nor{\triangleleft}

\def\ker#1{{\rm ker}(#1)}
\def\norm#1#2{{\bf N}_{#1}(#2)}
\def\cent#1#2{{\bf C}_{#1}(#2)}

\let\phi=\varphi
\def\sbs{\subseteq}

 \newcommand{\6}{^}

\newcommand{\Irr}{\operatorname{Irr}}

\newtheorem{lem}[subsection]{Lemma}
\newtheorem{cor}[subsection]{Corollary}
\newtheorem{thm}[subsection]{Theorem}
\newtheorem{prop}[subsection]{Proposition}

\newtheorem{con}[subsection]{Conjecture}

\newtheorem*{thm*}{Theorem}
\newtheorem*{thmA}{Theorem A}
\newtheorem*{corB}{Corollary B}

\theoremstyle{definition}
\newtheorem{rem}[subsection]{Remark}
\theoremstyle{definition}

\theoremstyle{definition}

\theoremstyle{definition}

\theoremstyle{definition}

\def\Q{\mathbb Q}

\begin{document}

\keywords{Character degrees, principal blocks, two primes, fields of values}

\subjclass[2010]{20C15, 20C30, 20C33}

\thanks{The work of the second named author is funded by grant
PID2022-137612NB-I00 funded by MCIN/AEI/ 10.13039/501100011033 and ERDF ``A way of making Europe''
and grant CIDEIG/2022/29 funded by Generalitat Valenciana. The first and third named authors acknowledge the support of the
 Istituto Nazionale di Alta Matematica (INdAM), as part of the GNSAGA}

\author{Eugenio Giannelli}
\address[E. Giannelli]{Dipartimento di Matematica e Informatica,
V. Morgagni 67, Firenze, Italy.}
\email{eugenio.giannelli@unifi.it}

\author{J. Miquel Mart\'inez}
\address[J. M. Mart\'inez]{Departament de Matem\`atiques, Universitat
  de Val\`encia, 46100 Burjassot, Val\`encia.}
\email{josep.m.martinez@uv.es}

\author{Carolina Vallejo}
\address[C. Vallejo]{Dipartimento di Matematica e Informatica,
V. Morgagni 67, Firenze, Italy.}
\email{carolina.vallejorodriguez@unifi.it}

\title[Cyclotomic character fields and sets of primes]{Cyclotomic character fields and sets of primes}

\date{\today}

\begin{abstract} Let $\pi=\{ 2, q \}$ where $q$ is an odd prime.
Let $G$ be a finite group of order divisible by a prime $p \in \pi$. We show that the principal $p$-block of $G$ contains a nontrivial irreducible character
of degree not divisible by $2$ nor $q$ and with field of values contained in the $q$th cyclotomic extension. 
This statement simultaneously provides a principal block version of results of Navarro--Tiep and Giannelli--Hung--Schaeffer Fry--Vallejo.
 \end{abstract}

\maketitle

\section*{Introduction}  

Let $G$ be a finite group and let $p$ be a prime divisor of the order of $G$. It is well-known, and elementary to prove, that $G$ possesses some nontrivial irreducible character $\chi$ whose degree $\chi(1)$ is coprime to $p$. 
If instead we want to prove the existence of a certain $\chi$ with degree coprime to $p$ and some restrictions on its field of values $\Q(\chi)$, where $\Q(\chi)$ is the extension of the rationals generated by the values of $\chi$, the use of more sophisticated techniques is often unavoidable. For instance, relying on the Classification of Finite Simple Groups, G. Navarro and P. H. Tiep \cite{NT06, NT08} proved that we can always
find some $\chi \in \irr G$ of degree coprime to $p$ with $\Q(\chi)\sbs \Q(\xi_p)$, where $\xi_p$ denotes a primitive $p$th root of unity.  
Such characters are known to heavily influence the structure of $G$. 
For instance,  \cite[Theorem C]{NT10} provides a criterion for the existence of a normal $p$-complement, 
formulated in terms of properties of characters
of  $\chi \in \irr G$ that have degree coprime to $p$ and satisfy $\Q(\chi)\sbs \Q(\xi_p)$.

In this note, we extend the main result of \cite{NT06} by showing that we can further choose $\chi$ as above to belong to the principal $p$-block of $G$. 
We obtain this refinement as a consequence of a more general result concerning characters with restricted field of values lying in principal blocks with degree coprime to a pair of primes $\pi$. 
In the following, for a set of primes $\pi$, we say that a character $\chi$ of $G$ has $\pi'$-degree if $(\chi(1), r)=1$ for every $r \in \pi$.

\begin{thmA} 
Let $G$ be a finite group, let $\pi=\{ 2, q \}$ where $q$ is a prime, and let $p\in \pi$. If $p$ divides the order of $G$, then the principal $p$-block of $G$ admits a non-trivial irreducible character $\chi$ of $\pi'$-degree with $\Q(\chi)\sbs \Q(\xi_q)$.
\end{thmA}

It is worth noting that, usually, the extension of results from a single prime $p$ to a set of primes $\pi$ fails to hold unless additional hypotheses, such as the $\pi$-separability of the group or the existence of a Hall $\pi$-subgroup, are assumed.
Notably, Theorem A holds in full generality, without any of these extra assumptions. We remark that the character $\chi$ identified by Theorem A may be unique. This phenomenon is not rare; for instance, it occurs for ${\sf A}_5$ when $q=3$, for ${\sf A}_6$ when $q=5$, for ${\sf M}_{23}$  when $q \in \{ 3, 5 , 7\}$, etc. In such cases, $\chi$ is actually rational and lies in the intersection of the principal $2$-block and $q$-block.

We also note that Theorem A provides a principal block analogue to the main result of \cite{GHSV21}. In Remark \ref{rem: error} below, we correct a minor oversight in the proof of \cite[Theorem D]{GHSV21}.

We underline that the hypothesis $2 \in \pi$ is necessary in Theorem A. We refer the reader to Section \ref{sec:comments} for a complete discussion of this remark. 
Furthermore, Theorem A cannot be generalized from a multiple-prime perspective, as the statement fails to hold when $|\pi| \ge 3$ (for instance take $A_5$ with $\pi=\{2,3,5\}$).

For a set $\pi$ of odd primes, it has recently been proven in \cite{GH26} that whenever $G$ admits a nontrivial Hall $\pi$-subgroup, $G$ possesses a nontrivial $\chi$ of $\pi'$-degree with values in $\mathbb Q(\xi_r)$ for some $r \in \pi$. One might speculate that such a character can be further chosen to lie in the principal $r$-block of $G$, but we do not aim to pursue this statement here.

As mentioned above, the following is an immediate consequence of Theorem A and a generalization of the main result of \cite{NT06}.
  
\begin{corB} Let $G$ be a group of order divisible by a prime number $p$. Then there is some nontrivial $\chi \in \irr G$ of degree coprime to $p$ with $\Q(\chi)\sbs \Q(\xi_p)$ lying in the principal $p$-block of $G$.
 \end{corB}

In Section \ref{sec:1}, we reduce Theorem A to a problem concerning finite simple groups. The proofs of Theorem A and Corollary B are then completed in Section \ref{sec:2}, using the Classification of Finite Simple Groups.
Problems regarding two primes and principal blocks represent an active area of research; we further discuss recent results in this direction and their relation to our work in Section \ref{sec:comments}.

 \subsection*{Acknowledgments} Part of this work was done while the second named author visited the International Center of Mathematics at SUSTech, Shenzhen. He is grateful to Zhicheng Feng and the whole institute for their kind hospitality. The authors also wish to thank Mandi Schaeffer Fry for a thorough read of the manuscript and for indispensable conversations on the topic of this note.
 
\section{A reduction theorem}\label{sec:1}

The purpose of this section is to show that in order to prove Theorem A it is enough to show that the statement holds for the simple non-abelian groups.

Below we recall some folklore about principal blocks.  For a prime $p$, we will denote by $B_p(G)$ the principal $p$-block of the group $G$, that is the only $p$-block containing the trivial character ${\bf 1}_G$, and by $\irr{B_p(G)} $ the set of irreducible characters of $G$ belonging to $B_p(G)$. 
We will denote by ${\bf O}_p (G)$ and ${\bf O}_{p'}(G)$, respectively, the largest normal $p$-subgroup and $p'$-subgroup of $G$. As every $\chi\in \irr{B_p(G)}$ contains ${\bf O}_{p'}(G)$ in its kernel \cite[Theorem 6.10]{Nav98}, we can see $\irr{B_p(G)}\sbs \irr{G/{\bf O}_{p'}(G)}$.

\begin{thm}[Fong]\label{thm:Fong} Let $G$ be a finite group and $p$ be a prime. Suppose that $G$ is $p$-solvable. Then 
$\irr{B_p(G)}=\irr {G/{\bf O}_{p'}(G)}$.
\end{thm}
\begin{proof} As mentioned above $\irr{B_p(G)}\sbs \irr{G/{\bf O}_{p'}(G)}$. 
By \cite[Lemma 3A]{Fon61}, we have that $\irr{G/{\bf O}_{p'}(G)}=\irr{B_p(G/{\bf O}_{p'}(G))}\sbs \irr{B_p(G)}$, and the equality follows. 
\end{proof}

\begin{thm}[Alperin--Dade]\label{thm:Alperin-Dade}
Let $G$ be a finite group and $p$ a prime. 
Suppose that $N$ is a normal subgroup of $G$, with $G/N$ a $p'$-group.
Let $P \in \syl p G$ and assume that $G=N\cent GP$. Then restriction of characters defines
a natural bijection between $\irr{B_p(G)}$ and $\irr{B_p(N)}$.
\end{thm}

\begin{proof}
The case where $G/N$ is solvable was proved in \cite{Alp76} and the general case
in \cite{Dad77}
\end{proof}

\begin{lem}\label{lem:covering} 
Let $G$ be a finite group and $p$ a prime. 
Let $N \nor G$ and $P \in \syl p N$. Write $E=N \cent G P$. Suppose that $p$ does not divide the order of $G/N$. Then
for all $\theta \in \irr{B_p(E)}$, every irreducible constituent of $\theta^G$ belongs to $B_p(G)$.
\end{lem}
\begin{proof} 
By the Frattini argument $G=N \norm G P$. In particular, $E\nor G$.
By \cite[Theorem 4.14 and Problem 4.2]{Nav98} and Brauer's third main theorem \cite[Theorem 6.7]{Nav98}, we have that $B_p(E)$ induces $B_p(G)$. By \cite[Theorem 9.19, Lemma 9.20 and Theorem 9.26]{Nav98}, $B_p(G)$ is actually the only $p$-block of $G$ covering $B_p(E)$.  In fact, any block $B$ of $G$ covering $B_p(E)$ has defect group $P$, then $B$ is regular with respect to $E$ and $B=B_p(E)^G=B_p(G)$. The result now follows from \cite[Theorem 9.2]{Nav98}.
\end{proof}

We are ready to prove a reduction theorem for Theorem A. Recall that for a positive integer $n$, we denote by $\xi_n$ a primitive $n$th root of unity. We denote by ${\bf O}^p(G)$ and ${\bf O}^{p'}(G)$, respectively, the smallest normal subgroups
 of $G$ with quotient a $p$-group and a $p'$-group.

\begin{thm}\label{reduction} 
Let $\pi=\{ 2, q \}$ where $q$ is an odd prime and $p\in \pi$. Let $G$ be a finite group of order divisible by $p$.
Suppose that the principal $p$-block of every finite simple non-abelian group $S$ of order divisible by $p$ possesses a nontrivial character of $\pi'$-degree with values in $\mathbb Q(\xi_q)$. 
Then there exists some ${\bf 1}_G\neq \chi \in \irr {B_p(G)}$ with $\pi'$-degree and $\Q(\chi)\sbs \mathbb Q(\xi_q)$.
\end{thm}
\begin{proof}  
Let $G$ be a counterexample of minimal order to the statement.
Notice that no counterexample can be $p$-solvable, by Fong's theorem \ref{thm:Fong} and the main result of \cite{GHSV21}.
By hypothesis, $G$ is not simple. If $1<N\nor G$ then $\Irr(B_0(G/N))\sbs\Irr(B_0(G))$. Since $G$ is a minimal counterexample, $|G/N|$ cannot be divisible by $p$. In particular ${\bf O}^p(G)=G$. 

Let $N={\bf O}^{p'}(G)$.
We have that ${\bf O}^p(N)\nor G$ and $G/{\bf O}^p(N)$ is $p$-solvable and hence ${\bf O}^p(N)=N$. 
If $M \nor G$ and $M<N$, then $p$ divides $|N/M|$. In particular, by the minimality of $G$ as a counterexample $M=1$.
We have that $N$ is a minimal normal subgroup of $G$. 
In fact, $N$ must be a semisimple group. 
Let $1\neq P \in \syl p N \subseteq \syl p G$ and consider $E=N \cent G P$. 
By Lemma \ref{lem:covering}, ${\rm Irr}(G/E)\subseteq {\rm Irr}_{p'}(B_p)$. By hypothesis, the only irreducible character of $G/E$ with $\pi'$-degree and values in $\Q(\xi_q)$ is ${\bf 1}_{G/E}$, hence $G/E$ is a group of odd order using \cite[Theorem A]{GHSV21}. By using the Feit--Thompson odd order theorem, $G/E$ is solvable and every composition factor is a cyclic group of order $r \notin \pi$. 

If $E<G$, then by hypothesis $p$ divides $|E|$ and $E$ is not a counterexample to the statement, so there is some ${\bf 1}_E\neq \psi \in \irr{B_p(E)}$ of $\pi'$-degree with values in $\Q(\xi_q)$. Using \cite[Lemma 2.2]{GHSV21} and taking into account that $G/E$ is a $\pi'$-group we can produce $\chi \in \irr{G}$ over $\psi$ with $\pi'$-degree and values in $\Q(\xi_q)$. By Lemma \ref{lem:covering},
$\chi$ lies in $B_p$ and we get a contradiction to $G$ being a counterexample. Hence $E=G$. 

The Alperin--Dade theorem \ref{thm:Alperin-Dade} then tells us that $N$ has no nontrivial $\pi'$-degree character in $\irr{B_p(N)}$ with values in $\Q(\xi_q)$. 
By minimality of $G$ as a counterexample $G=N$.
Recall that $G$ is semisimple of order divisible by $p$. Let $S$ be a minimal normal subgroup of $G$, then $G\cong S^t$ for some $t \geq 1$. By hypotheses $S$
possesses 
a nontrivial character $\theta \in \irr{B_p(S)}$ of $\pi'$-degree with values in $\mathbb Q(\xi_q)$. Hence ${\bf 1}_G\neq \psi =\theta\times \cdots \times \theta \in \irr{B_p(G)}$ has $\pi'$-degree and $\Q(\psi)=\Q(\theta)\sbs \mathbb Q(\xi_q)$, a contradiction. 
\end{proof}

\section{Simple non-abelian groups}\label{sec:2}

The aim of this section is to prove that Theorem A holds whenever $G$ is a simple non-abelian group. In particular, at the end of this section we complete the proofs of Theorem A and Corollary B.

\subsection{Alternating Groups}
Let $n\in\mathbb{N}$ be greater than or equal to $5$ and let $p$ be a prime number such that $p\leq n$. The aim of this section is to show that $\mathrm{Irr}_{\{p,q\}'}(B_p(A_n))$ admits a non-trivial irreducible character $\chi$ such that $$\mathbb{Q}(\chi)\subseteq \mathbb{K},\ \text{for some}\ \mathbb{K}\in\{\mathbb{Q}(\xi_p), \mathbb{Q}(\xi_q)\},$$ for all prime numbers $q\neq p$. 
This is an extension of the work done in \cite[Theorem D]{GHSV21} where we complete a similar taks without taking blocks into account. While writing the present article we realised that the proof of \cite[Theorem D]{GHSV21} contains a small mistake. We take the occasion to correct it (see Remark \ref{rem: error} below). 

We recall that irreducible character of $S_n$ are naturally labelled by partitions of $n$. Setting $\mathcal{P}(n)$ as the set of partition of $n$, we let $\chi^\lambda$ denote the irreducible character of $S_n$ corresponding to $\lambda\in\mathcal{P}(n)$. 
We remark that the character $(\chi^\lambda)_{A_n}$ is irreducible if and only if $\lambda\neq \lambda'$, where $\lambda'$ is the conjugate of $\lambda$.
We refer the reader to \cite{OlssonBook} for a detailed account of the basic facts concerning the representation theory of symmetric groups and the corresponding combinatorics of partitions. Given $\lambda\in\mathcal{P}(n)$ 
 and $i,j\in\mathbb{N}$ we denote by
$H_{ij}(\lambda)$ the \textit{hook of} $\lambda$
corresponding to node $(i,j)$, and we denote its length by $h_{ij}(\lambda)$. 
For $e\in\mathbb{N}$, we let
$\mathcal{H}^{e}(\lambda)$ be the set consisting of all those nodes
$(i,j)$ of $\lambda$ such that $e$ divides $h_{ij}(\lambda)$.
Moreover, we let $C_e(\lambda)$ denote the $e$-core of $\lambda$.
For a fixed prime number $p$ we write $\mathcal{P}_{p'}(n)$ for the subset of $\mathcal{P}(n)$ consisting of those partitions $\lambda$ of $n$ such that $(\chi^\lambda(1), p)=1$. 

We start by stating a few preliminary statements that will be often used in our arguments.  
\begin{lem}\label{lem: olsson1}
Let $p$ be a prime and let $n=ap^k+r$, where $1\leq a\leq p-1$ and $0\leq r<p^k$. Let $\lambda\in\mathcal{P}(n)$. Then $\lambda\in\mathcal{P}_{p'}(n)$ if and only if $C_{p^k}(\lambda)\in\mathcal{P}_{p'}(r)$. 
\end{lem}
\begin{proof}
This is an immediate consequence of \cite[Proposition 6.4]{OlssonBook}. 
\end{proof}

The following is an easy consequence of the combinatorial description of $p$-blocks of symmetric groups. 

\begin{lem}\label{lem: olssonBlocks}
Let $p$ be a prime, let $n=pw+r$ for some $0\leq r\leq p-1$, and let $\lambda\in\mathcal{P}(n)$. Then $\chi^\lambda\in \mathrm{Irr}(B_p(S_n))$ if and only if $C_p(\lambda)=(r)$.  
In particular, if $C_p(\lambda)=(r)$ then every irreducible constituent of $(\chi^\lambda)_{A_n}$ lies in $\mathrm{Irr}(B_p(A_n))$.
\end{lem}

We are ready to focus on our main goal. We start by settling the case where $q>n$.

\begin{lem}\label{alterna q more than n}
Let $p,q$ be primes and let $n\in\mathbb{N}$ be such that $p\leq n<q$. Then the set $\mathrm{Irr}_{\{p,q\}'}(B_p(A_n))$ admits a rational non-trivial irreducible character.
\end{lem}
\begin{proof}
Let $n=a+pw$ for some $a,w\in\mathbb{N}$ such that $0\leq a\leq p-1$. Moreover, let 
$$pw=\sum_{j=1}\6kb_jp\6{j},$$
be the $p$-adic expansion of $pw$. Here $k\in\mathbb{N}_{\geq 1}$ and $b_1,\ldots, b_k\in [0,p-1]$ with $b_k\neq 0$.

We proceed by distinguishing three cases. 

\noindent$\bullet$ If $pw>p\6k$ then let $\chi:=(\chi\6\lambda)_{A_n}$, where 
$$\lambda=\begin{cases} (n-p\6k,2, 1\6{p\6k-2}) & \mathrm{if}\ n\neq 2p\6k,\\
\\
(p\6k, 1\6{p\6k}) & \mathrm{if}\ n=2p\6k.
\end{cases}
$$
We first notice that $\lambda\neq\lambda'$. This shows that $\chi\in\mathrm{Irr}(A_n)$ and that $\mathbb{Q}(\chi)=\mathbb{Q}$. It is also easy to see that $\lambda\notin\{(n), (1\6n)\}$ and therefore that $\chi$ is not the trivial character. We observe that
$h_{2,1}(\lambda)=p\6k$ and that $\lambda\smallsetminus H_{2,1}(\lambda)=(n-p\6k)$ is the trivial partition of $n-p\6k$. Using Lemma \ref{lem: olsson1} this shows that $(p,\chi(1))=1$. Moreover, it also follows that $C_p(\lambda)=(a)$. Since $n<q$ we clearly have that $(q,\chi(1))=1$ and we conclude that $\chi\in\mathrm{Irr}_{\pi'}(B_p(A_n))$. 

\noindent$\bullet$ Let us now assume that $pw=p\6k$ and let $\chi:=(\chi\6\lambda)_{A_n}$, where 
$$\lambda=\begin{cases} (n-(a+1), a+1) & \mathrm{if}\ n\neq 2p-1,\\
\\
(p-1, 1\6{p}) & \mathrm{if}\ n=2p-1.
\end{cases}
$$
In this case it is very easy to verify that $\chi\in\mathrm{Irr}_{\pi'}(B_p(A_n))$. We omit the full details of this verification. We just notice that when $n\neq 2p-1$ we have that $H_{1,1}(\lambda)=p\6k$. This shows that $C_p(\lambda)=(a)$, and that $(p,\chi(1))=1$ by Lemma \ref{lem: olsson1}. 
\end{proof}

\begin{rem}\label{rem: error}
In \cite[Theorem D]{GHSV21} for every pair of prime numbers $p,q$ and for every integer $n$ we determine a non-trivial irreducible character $\chi$ of $\mathrm{Irr}_{\{p,q\}'}(A_n)$ with small field of values. At the very end of the proof (in the last five lines) we deal with the case where $p\leq n<q$ and we consider the partition $\lambda=(n-(a+1),a+1)$. We also claim that $p$ does not divide $\chi\6\lambda(1)$. This claim is false, as shown for instance by taking $p=3$, $n=14$ and $\lambda=(11,3)$. 
This error can be fixed by replacing the final part of the proof of \cite[Theorem D]{GHSV21} with the above 
Lemma \ref{alterna q more than n}.
\end{rem}

We are now ready to prove our main result for alternating groups. 

\begin{thm}\label{thm:alternating}
Let $\pi=\{p,q\}$ be a set containing two distinct prime numbers, let $n\in\mathbb{N}_{\geq 5}$ and let $r\in \pi$.
If $r\leq n$ then there exists a non-trivial character $\chi\in\mathrm{Irr}_{\pi'}(B_r(A_n))$ such that 
$$\mathbb{Q}(\chi)\subseteq \mathbb{K},\ \text{for some}\ \mathbb{K}\in\{\mathbb{Q}(\xi_p), \mathbb{Q}(\xi_q)\}.$$
\end{thm}
\begin{proof}
If both primes $p$ and $q$ are larger than $n$ then the statement holds vacuosly as no $r\in\pi$ can possibly satisfy $r\leq n$. 
Let us now assume that exactly one of the primes in $\pi$ is larger than $n$. Without loss of generality we can assume that $p\leq n<q$. In this case $r=p$ and the statement holds by Lemma \ref{alterna q more than n}. 
Hence, we can now focus on the case where both $p,q\leq n$. 
In \cite[Theorem D]{GHSV21} we have shown the existence of a non-trivial $\chi\in\mathrm{Irr}_{\pi'}(A_n)$. We are now going to verify that such a character $\chi$ actually lies in both $B_p(A_n)$ and $B_q(A_n)$. This is of course enough to prove the theorem. 

Let $n=\sum_{i=1}\6{t}a_ip\6{m_i}+a_0=\sum_{j=1}\6{r}b_jq\6{k_j}+b_0$ be the
$p$-adic and respectively $q$-adic expansions of $n$. Here
$m_1>m_2>\cdots >m_t\geq 1$ and $k_1>k_2>\cdots k_r\geq 1$. Without
loss of generality, possibly exchanging the roles of the primes $p$ and $q$, we can assume that $b_1q\6{k_1}<a_1p\6{m_1}$. 

We first
consider the case where $n$ does not satisfy the following equalities: $a_1p\6{m_1}=n=b_1p\6{k_1}+1$. 
In this setting we take $\lambda\in\mathcal{P}(n)$ to be defined by:
$$\lambda=(n-b_1q^{k_1}, n-a_1p^{m_1}+1, 1^{b_1q^{k_1}-(n-a_1p^{n_1}+1)}),$$
and we let $\chi=(\chi\6\lambda)_{A_n}$. We observe that $h_{1,1}(\lambda)=a_1p\6{m_1}$ and that $h_{2,1}(\lambda)=b_1q\6{k_1}$. It follows that 
$$C_{p\6{m_1}}(\lambda)=(n-a_1p\6{m_1})\ \ \text{and}\ \ C_{q\6{k_1}}(\lambda)=(n-b_1q\6{k_1}),$$
are trivial partitions of $n-a_1p\6{m_1}$ and $n-b_1q\6{k_1}$, respectively. 
This shows that $C_p(\lambda)=C_p(C_{p\6{m_1}}(\lambda))=(a_0)$ and that $C_q(\lambda)=C_q(C_{q\6{k_1}}(\lambda))=(b_0)$. In particular we have that $\chi\in\mathrm{Irr}(B_p(A_n))\cap \mathrm{Irr}(B_q(A_n))$ by Lemma \ref{lem: olssonBlocks}. In the proof of \cite[Theorem D]{GHSV21} it is shown that $1\neq \chi\in\mathrm{Irr}_{\pi'}(A_n)$ and that $\mathbb{Q}(\chi)=\mathbb{Q}$. We conclude that $\chi$ is the irreducible character we were looking for. 

We can now assume that $a_1p\6{m_1}=n=b_1p\6{k_1}+1$. To ease the notation we set $a:=a_1, b:=b_1, m:=m_1$ and $k:=k_1$. We split this second case into several subcases. 

\noindent$\bullet$ If $b\geq 3$ we consider $\lambda=(1+(b-1)q\6k, 1\6{q\6k})$ and we let $\chi=(\chi\6\lambda)_{A_n}$. We observe that $h_{1,1}(\lambda)=ap\6{m}$, $h_{1,2}(\lambda)=(b-1)q\6{k}$ and $h_{2,1}(\lambda)=q\6{k}$. It follows that 
$$C_p(\lambda)=C_p(C_{p\6{m}}(\lambda))=\emptyset,\ \ \text{and}\ \ C_q(\lambda)=C_q(C_{q\6{k}}(\lambda))=(1).$$
We deduce that $\chi\in\mathrm{Irr}(B_p(A_n))\cap \mathrm{Irr}(B_q(A_n))$ by Lemma \ref{lem: olssonBlocks}. Again, using the same ideas adopted in the proof of \cite[Theorem D]{GHSV21} we conclude that $1\neq \chi\in\mathrm{Irr}_{\pi'}(B_r(A_n))$ for any $r\in\pi$ and that $\mathbb{Q}(\chi)=\mathbb{Q}$. 

\noindent$\bullet$ Assume now that $b\leq 2$ and that $a\geq 3$. In this case let $\mu=((a-1)p\6m, 2, 1\6{p\6m-2})$ and we let $\chi=(\chi\6\mu)_{A_n}$. Notice that $h_{1,1}(\mu)=bq\6{k}$, $h_{1,2}(\lambda)=(a-1)p\6{m}$ and $h_{2,1}(\lambda)=p\6{m}$. It follows that 
$$C_p(\mu)=C_p(C_{p\6{m}}(\mu))=\emptyset,\ \ \text{and}\ \ C_q(\mu)=C_q(C_{q\6{k_1}}(\mu))=(1).$$
We deduce that $\chi\in\mathrm{Irr}(B_p(A_n))\cap \mathrm{Irr}(B_q(A_n))$ by Lemma \ref{lem: olssonBlocks}.
From the proof of \cite[Theorem D]{GHSV21} we have that $1\neq \chi\in\mathrm{Irr}_{\pi'}(A_n)$ and that $\mathbb{Q}(\chi)=\mathbb{Q}$. This shows that 
$1\neq \chi\in\mathrm{Irr}_{\pi'}(B_r(A_n))$ for any $r\in\pi$ and that $\mathbb{Q}(\chi)=\mathbb{Q}$. 

\noindent$\bullet$ We can now assume that $a,b\in\{1,2\}$. In particular we need to study the problem for any pair of integers $(a,b)\in\{(1,1), (1,2), (2,1)\}$. We start by treating at once the cases $(a,b)=(1,2)$ and $(a,b)=(2,1)$. 
In order to do this, let $\zeta\in\mathcal{P}(n)$ be defined as follows:
$$\zeta=\begin{cases} (1+q\6k, 1\6{q\6k}) & \mathrm{if}\ (a,b)=(1,2),\\
\\
(p\6m, 2, 1\6{p\6m-2}) & \mathrm{if}\ (a,b)=(2,1).
\end{cases}
$$
Moreover, we denote by $\chi\6+$ and $\chi\6-$ the irreducible constituents of $(\chi^\zeta)_{A_n}$. 
From the proof of \cite[Theorem D]{GHSV21} we have that $\mathrm{Irr}_{\pi'}(A_n)=\{1, \chi\6+, \chi\6-\}$ and that 
$$\mathbb{Q}(\chi\6\pm)\subseteq\begin{cases} \mathbb{Q}(\xi_p) & \mathrm{if}\ (a,b)=(1,2),\\
\\
\mathbb{Q}(\xi_q) & \mathrm{if}\ (a,b)=(2,1).
\end{cases}
$$
It is not difficult to compute the $p$-core and the $q$-core of $\zeta$. In fact we have that
$C_p(\zeta)=\emptyset$ and $C_q(\zeta)=(1)$. As usual this implies that both $\chi\6+$ and $\chi\6-$ lie in both the principal $p$-block and the principal $q$-block of $A_n$. This concludes this case. 

Finally, we assume that $(a,b)=(1,1)$. In other words we have that $p\6m=q\6k+1$. This implies that $2\in\pi$. 
In this case we let $\eta\in\mathcal{P}(n)$ be defined as follows:
$$\eta=\begin{cases} (2\6{m-1}, 2, 1\6{2\6{m-1}-2}) & \mathrm{if}\ p=2,\\
\\
(1+2\6{k-1}, 1\6{2\6{k-1}}) & \mathrm{if}\ q=2.
\end{cases}
$$
Moreover, we denote by $\chi\6+$ and $\chi\6-$ the irreducible constituents of $(\chi^\eta)_{A_n}$. 
From the proof of \cite[Theorem D]{GHSV21} we have that $\mathrm{Irr}_{\pi'}(A_n)=\{1, \chi\6+, \chi\6-\}$ and that 
$$\mathbb{Q}(\chi\6\pm)\subseteq\begin{cases} \mathbb{Q}(\xi_p) & \mathrm{if}\ q=2,\\
\\
\mathbb{Q}(\xi_q) & \mathrm{if}\ p=2.
\end{cases}
$$
Computing the $p$-core and the $q$-core of the partition $\eta$ (in both the possible cases), we easily deduce that $\chi\6+$ and $\chi\6-$ lie in both the principal $p$-block and the principal $q$-block of $A_n$. This concludes the proof. 
\end{proof}

 \subsection{Sporadic and Lie type groups}

 \begin{prop}\label{prop:simple-exceptions} Let $\pi=\{ 2, q\}$ where $q$ is an odd prime.
 Let $S$ be a sporadic group, the Tits group or a simple group of Lie type with exceptional Schur multiplier.  Let $p\in\pi$  and assume $(|S|,p)\neq 1$. Then there is ${\bf 1}_S\neq \chi\in\Irr(B_p)$ with $\pi'$-degree and such that $\mathbb{Q}(\chi)\sbs\mathbb{Q}(\xi_q)$.
 \end{prop}
 \begin{proof} This can be confirmed using \cite{GAP}.
 \end{proof}

We establish some notation to deal with the remaining simple groups of Lie type. Let $\mathbf{G}$ be a connected reductive group defined over $\overline{\mathbb{F}}_\ell$ for some prime $\ell$. A simple group of Lie type $S$ occurs as $S=G/\zent G$ where $G=\mathbf{G}^F$ is the set of fixed points of a Steinberg morphism $F$, and $\mathbf{G}$ is simple of simply connected type. Let $\sigma:\mathbf{G}\hookrightarrow\widetilde{\mathbf{G}}$ be a regular embedding as in \cite[(15.1)]{CabEng}, where $\widetilde{\mathbf{G}}$ is also a connected reductive group defined over $\overline{\mathbb{F}}_\ell$, and with $\zent{\widetilde{\mathbf{G}}}$ connected. Let $\mathbf{G}^*$ and $\widetilde{\mathbf{G}}^*$ be the dual groups of $\mathbf{G}$ and $\widetilde{\mathbf{G}}$ respectively and denote by $\sigma^*:\widetilde{\mathbf{G}}^*\rightarrow\mathbf{G}^*$ the surjective morphism from \cite[(15.1*)]{CabEng}. Note that $\ker{\sigma^*}\sbs\zent{\widetilde{\mathbf{G}}^*}$. Write $\widetilde{G}=\widetilde{\mathbf{G}}^F$ and note that $G\nor \widetilde{G}$. Denote also $G^*=\mathbf{G}^*$ and $\widetilde{G}^*=(\widetilde{\mathbf{G}}^*)^F$. The surjective morphism $\sigma^*$ also induces a surjection $\sigma^*:\widetilde{G}^*\rightarrow G^*$ with central kernel.

 \begin{prop}\label{prop: lie type} Let $\pi=\{ 2, q\}$ where $q$ is an odd prime.
 Let $S$ be a finite simple group of Lie type defined over a field of characteristic $\ell$ not considered in Proposition \ref{prop:simple-exceptions}, let $p\in\pi$  and assume $(|S|,p)\neq 1$. Then there is ${\bf 1}_S\neq \chi\in\Irr(B_p)$ with $\pi'$-degree and such that $\mathbb{Q}(\chi)\sbs \mathbb{Q}(\xi_q)$.
 \end{prop}
 \begin{proof} By Proposition \ref{prop:simple-exceptions} we can assume that $S$ has not an exceptional Schur multiplier nor is the Tits group. If $\ell \notin \pi$ by \cite[Proposition 3.7]{NRS22} we can actually choose $\chi$ to be rational-valued. Hence we assume $\ell\in \pi$. If $\ell=p$, then by Dagger--Humphreys \cite[Theorem 3.3]{Cab18}, $\Irr(B_p(S))=\Irr(S)-\{\mathsf{St}_S\}$,
so $\Irr_{p'}(B_p(S))=\Irr(S)$. We are therefore done  by \cite[Proposition 4.5]{GHSV21}.

We are left to deal with the situation in which $\ell\neq p$, so  $\{ \ell, p\}=\pi=\{ 2, q \}$. We use now the notation introduced before the statement of this result. In this case, \cite[Proposition 4.5]{GHSV21} constructs a semisimple character $\chi_s$ associated to a semisimple element $s$ of the dual group $G^*$ such that $s$ has $p$-power order and $\mathbb{Q}(\chi_s)\sbs\mathbb{Q}(\xi_q)$. Let $\widetilde{s}\in\widetilde{G}^*$ be a preimage of $s$ by $\sigma^*$ of $p$-power order (if $t\in(\sigma^*)^{-1}(s)$, writing $t=t_pt_{p'}$ where $t_p$ is the $p$-part of $t$ and $t_{p'}$ si the $p'$-part of $t$, we have that $\sigma^*(t_{p'})=1$ so $\sigma^*(t_p)=s$ and we can indeed take such a preimage). By \cite[Corollary 3.4]{Hiss90} we have that $\chi_{\widetilde{s}}$ lies in $B_p(\widetilde{G})$. Now \cite[Proposition 15.6]{CabEng} implies that $\chi_s$ is a constituent of $(\chi_{\widetilde{s}})_{G}$ so by \cite[Theorem 9.2]{Nav98} we have that $\chi_s\in\Irr(B_p(G))$.
Since $S=G/\zent G$ and the characters $\chi_s$ constructed in \cite[Proposition 4.5]{GHSV21} contain $\zent G$ in their kernel, we have $\chi_s\in\Irr(B_p(S))$ by \cite[Lemma 17.2]{CabEng}, and this completes the proof.
 \end{proof}

\subsection{Main results}

We can now prove Theorem A and Corollary B from the Introduction that we restate here for the reader's convenience. 

\begin{thm}\label{thm:ThmA}   Let $\pi=\{ 2, q \}$ where $q$ is and odd a prime. Let $p \in \pi$ and $G$ be a finite group of order divisible by $p$. 
Then there is some ${\bf 1}_G\neq \chi \in \irr {G}$ in the principal $p$-block of $G$ with $\pi'$-degree and $\Q(\chi)\sbs \Q(\xi_q)$.
\end{thm}

\begin{proof} By using the Classification of Finite Simple Groups and Theorem \ref{reduction}, the proof follows directly from Theorem \ref{thm:alternating} and Propositions \ref{prop:simple-exceptions} and \ref{prop: lie type}.
\end{proof}

\begin{cor} Let $G$ be a group of order divisible by a prime number $p$. Then there is some nontrivial $\chi \in \irr G$ of degree coprime to $p$ with $\Q(\chi)\sbs \Q(\xi_p)$ lying in the principal $p$-block of $G$.
 \end{cor}
\begin{proof}
If $p$ is odd, this is an immediate consequence of Theorem \ref{thm:ThmA}.
 If $p=2$ we know by \cite[Theorem B]{NT08} that a group $G$ of even order contains a nontrivial rational irreducible character $\chi$ of odd degree. 
 We show that the decomposition number $d_{\chi {\bf 1}_{G^0}}\neq 0$, where ${\bf 1}_{G^0}$ denotes the trivial $2$-Brauer character of $G$. This will show that $\chi \in \irr{B_2(G)}$.
 Since $\chi$ has odd degree and it is rational, if we consider its restriction to $G^0$ the $2$-regular elements of $G$, we must find some real irreducible constituent $\varphi$ of odd degree.
 By a theorem of Fong \cite[Theorem 2.30]{Nav98} we have that $\varphi={\bf 1}_{G^0}$.
\end{proof}

\section{Further comments} \label{sec:comments}

In this final section, we contextualize our work in terms of recent research on principal blocks and various primes, and discuss some variations of our main result.

To begin, we note that the requirement that the set $\pi$ contains the prime $2$ is necessary for the statement of Theorem A to hold. For instance, the only irreducible character $\chi$ of the Janko group $\mathsf{J}_4$ of $\{23, 43\}'$-degree satisfying $\mathbb{Q}(\chi) \subseteq \mathbb{Q}(\xi_{23})$ or $\mathbb{Q}(\chi) \subseteq \mathbb{Q}(\xi_{43})$ is the trivial character.

If we focus on any pair of primes $\pi$, and forget about restrictions on fields of values, we believe that the following should hold.

\begin{con}\label{con:twoprimes}
Let $\pi$ be a set containing at least two primes, and let $\ell \in \pi$. If $G$ is a finite group of order divisible by $\ell$, then the principal $\ell$-block of $G$ contains a nontrivial character of $\pi'$-degree.
\end{con}

Notice that Conjecture \ref{con:twoprimes} is a principal block version of the main result of \cite{GSV19}. Proceeding as in the proof of Theorem \ref{reduction}, the verification of this conjecture can be reduced to the case of non-abelian simple groups. The case of the alternating groups has already been treated in Theorem \ref{thm:alternating}, while the sporadic simple groups and the simple groups of Lie type with exceptional Schur multiplier can be verified computationally using GAP \cite{GAP}. Consequently, the main difficulty lies in proving Conjecture \ref{con:twoprimes} for the simple groups of Lie type with non-exceptional Schur multipliers. Using \cite[Lemmas 3.5 and 3.6]{NRS22} it is possible to reduce this proof to classical-type groups with the defining prime not in $\pi$.
The results in \cite[Theorem A]{Bar25}, although motivated by a conjecture of Navarro, Rizo, and Schaeffer Fry \cite[Conjecture A]{NRS22} concerning character intersection properties for principal blocks with respect to two primes, further reduce the proof of Conjecture \ref{con:twoprimes} to types $D$ and ${}^2D$.

It is worth noting that \cite[Conjecture A]{NRS22} does not seem to imply Conjecture \ref{con:twoprimes}. However, as mentioned in the Introduction of \cite{Bar25}, the current strategy to prove \cite[Conjecture A]{NRS22} relies on proving the existence
of a nontrivial character of $\pi'$-degree lying in the principal $\ell$-block for every $\ell\in \pi$ for the finite simple groups of Lie type.

\end{document}